\documentclass[11pt,letterpaper]{amsart}
\usepackage[T1]{fontenc}
\usepackage[utf8]{inputenc}
\usepackage{lmodern}
\usepackage{amsmath,amssymb,amsthm,mathtools,mathrsfs}
\usepackage{microtype}
\usepackage[margin=1.25in]{geometry}
\usepackage[hidelinks]{hyperref}
\usepackage{tikz}                 
\usetikzlibrary{arrows.meta}       
\usepackage{placeins}             
\hypersetup{
  pdftitle={Dyer's conjecture on bounded joins in weak order},
  pdfauthor={Yibo Gao, Yulin Peng, Hanlin Xu},
  pdfkeywords={Coxeter groups, weak order, Bruhat order, inversion sets}
}

\newtheorem{theorem}{Theorem}[section]
\newtheorem{prop}[theorem]{Proposition}
\newtheorem{lemma}[theorem]{Lemma}
\newtheorem{cor}[theorem]{Corollary}
\theoremstyle{definition}
\newtheorem{conjecture}[theorem]{Conjecture}
\theoremstyle{remark}
\newtheorem{remark}[theorem]{Remark}
\newtheorem{example}[theorem]{Example}

\DeclareTextFontCommand{\emph}{\color{blue}\em}
\DeclareMathOperator{\cl}{cl}
\DeclareMathOperator{\cone}{cone}
\newcommand{\R}{\mathbb R}
\newcommand{\RR}{\mathcal R}
\newcommand{\BC}{\mathscr B}
\newcommand{\ccl}{\cl_2}
\newcommand{\A}{\mathcal A}
\newcommand{\C}{\mathcal C}
\newcommand{\s}{\mathbf{s}}
\newcommand{\pp}{\Phi^+}
\newcommand{\covB}{\lessdot_B}
\newcommand{\covL}{\lessdot_L}
\newcommand{\eps}{\varepsilon}

\title{Bounded joins of biclosed sets}
\author{Yibo Gao}
\address{Beijing International Center for Mathematical Research, Peking University, Beijing, China}
\email{gaoyibo@bicmr.pku.edu.cn}
\author{Yulin Peng}
\address{Qiuzhen College, Tsinghua University, Beijing, China}
\email{pyl24@mails.tsinghua.edu.cn}
\author{Hanlin Xu}
\address{School of Mathematical Sciences, Peking University, Beijing, China}
\email{2401110030@stu.pku.edu.cn}
\date{September 19, 2026}

\begin{document}

\begin{abstract}
Dyer conjectured that the join of two biclosed sets of positive roots can be described by increasing Bruhat paths whose reflection labels belong to their union. We give a type-uniform proof of this conjecture for all finite Coxeter groups. More generally, for a $2$-coclosed set $C$ of positive roots contained in an inversion set, we show that its $2$-closure is an inversion set $I(w)$ and that the elements reachable from the identity using reflections labeled by roots in $C$ form exactly the set $[e,w]_B w^{-1}$. The proof combines Dyer's closure criteria with a root-selection argument.
\end{abstract}

\maketitle
\section{Introduction}\label{sec:intro}
Let $(W,S)$ be a Coxeter system with its standard real root system $\Phi=\pp\sqcup-\pp$. For $w\in W$, write $I(w)=\{\alpha\in\pp:w^{-1}\alpha\in-\pp\}$ for its \emph{(left) inversion set}. The \emph{right and left weak orders}, denoted by $\leq_R$ and $\leq_L$, respectively, are defined by prefixes and suffixes of reduced expressions: $u\le_R w$ if some reduced expression for $u$ is a prefix of a reduced expression for $w$, and $u\le_L w$ if some reduced expression for $u$ is a suffix of a reduced expression for $w$.

It is classical that $u\le_R w$ if and only if $I(u)\subseteq I(w)$, and that the weak order of a finite Coxeter group is a lattice. More generally, for an arbitrary Coxeter system, weak order is a complete meet-semilattice \cite[Theorem~3.2.1]{bjornerCombinatoricsCoxeterGroups2005}. For infinite Coxeter groups, the meet always exists, while the existence of a join requires an upper bound.

The description of weak order by containment of inversion sets suggests describing joins in terms of roots. A subset of $\pp$ is \emph{$2$-closed} if it contains every positive root in the nonnegative span of any two of its members. Its \emph{$2$-closure} is denoted by $\ccl$. A set is \emph{$2$-coclosed} if its complement in $\pp$ is $2$-closed, and \emph{$2$-biclosed} if both it and its complement are $2$-closed. Dyer proved that whenever $u\vee_R v$ exists, its inversion set is given by $I(u\vee_R v)=\ccl(I(u)\cup I(v))$, and the same description holds for arbitrary set of elements of $W$ with a common upper bound \cite[Theorem~1.5]{dyerWeakOrderCoxeter2019}.

Let $\BC(\pp)$ be the inclusion poset of $2$-biclosed subsets of $\pp$, called the \emph{extended weak order}. Its finite elements are precisely the inversion sets, so it contains the weak order on $W$ as an order ideal and coincides with it when $W$ is finite \cite[\S2.3]{dyerWeakOrderCoxeter2019}. Dyer conjectured that $\BC(\pp)$ is always a complete lattice, with joins given by the $2$-closure of unions \cite[\S2.5]{dyerWeakOrderCoxeter2019}. Barkley and Speyer classified biclosed sets in affine root systems and established the lattice property in types $\widetilde A$ and $\widetilde C$ \cite{barkleyCombinatorialBiclosedAffine2024}. They subsequently proved the complete lattice property and the closure formula for joins in all affine types \cite{barkleyAffineExtendedWeakOrder2023}. Beyond affine type, Barkley, Defant, Hersh, McCammond, McConville, and Speyer proved these assertions for the rank-three universal Coxeter group \cite{barkleyExtendedWeakOrderUniversal2026}. The conjecture studied here asks for a further description of these joins through increasing paths in the Bruhat graph.

We write $\le_B$ for the Bruhat order and $[e,w]_B=\{y\in W:y\le_B w\}$. For $C\subseteq\pp$, let $\RR(C)$ be the set of endpoints of paths starting at $e$ whose steps have the form $u\longrightarrow ut_\alpha$, with $\alpha\in C$ and $\ell(ut_\alpha)>\ell(u)$. Here $t_\alpha$ is the reflection corresponding to $\alpha$. We call $\RR(C)$ the \emph{reachable set} of $C$, and the \emph{Bruhat preclosure} of $C$ is $\overline C=\{\alpha\in\pp:t_\alpha\in\RR(C)\}$.  Dyer conjectured a description of the join of two biclosed sets using Bruhat preclosure. 

\begin{conjecture}[Dyer {\cite[\S2.8]{dyerWeakOrderCoxeter2019}}]\label{conj:dyer}
For $E,F\in\BC(\pp)$, the set $\overline{E\cup F}$ is the join of $E$ and $F$ in $\BC(\pp)$.
\end{conjecture}

We prove this conjecture whenever the union is contained in an inversion set. This includes all finite Coxeter groups and all pairs of inversion sets with a common upper bound in arbitrary Coxeter groups. Our main theorem applies more generally to any $2$-coclosed set $C$ that is \emph{bounded}, meaning that $C\subseteq I(x)$ for some $x\in W$. The case of arbitrary unbounded biclosed sets is not addressed here.

\begin{theorem}\label{thm:main}
Let $(W,S)$ be any Coxeter system with its standard real root system. Suppose that $C\subseteq\pp$ is $2$-coclosed and that $C\subseteq I(x)$ for some $x\in W$. Then there is a unique $w\in W$ such that $\ccl(C)=I(w)$. Moreover, $w\le_R x$ and
\begin{equation}\label{eq:main}
 \RR(C)=\RR(I(w))=[e,w]_B w^{-1},\qquad
 \overline C=\ccl(C)=I(w).
\end{equation}
\end{theorem}

\begin{cor}\label{cor:joins}
Let $(E_i)_{i\in J}$ be a family of sets in $\BC(\pp)$ such that $\bigcup_{i\in J} E_i$ is bounded.
Then $\bigvee_{i\in J} E_i$ exists in $\BC(\pp)$ and is given by
$\overline{\bigcup_{i\in J}E_i}$. In particular, if $\Phi^+$ is finite, then $\overline{\bigcup_{i\in J}E_i}$ is the join of $(E_i)_{i\in J}$ in $\BC(\pp)$.
\end{cor}

The formula $E\vee F=\overline{E\cup F}$ conjectured by Dyer is false
for general biclosed sets in infinite Coxeter groups, as the following example shows.

\begin{example}\label{ex:affine-a2-counterexample}
Let $W$ be a Coxeter group of type $\widetilde A_2$, with simple
reflections $s_0,s_1,s_2$. Set $E=I(s_2s_1)$, $F=I(s_0)$, and
$C=E\cup F=\{\alpha_0,\alpha_2,\alpha_1+\alpha_2\}$.
We abbreviate $i_1\cdots i_k=s_{i_1}\cdots s_{i_k}$.

The increasing path $e\to0\to02\to020$ shows that
$\alpha_0+\alpha_2\in\overline C$.
Also, $\alpha_1+\alpha_2\in C\subseteq\overline C$.
However, their sum $\gamma=\alpha_0+\alpha_1+2\alpha_2$, whose corresponding reflection is $t_\gamma=20102$, is a positive
root not in $\overline{C}$.
Indeed, all maximal decreasing paths from $20102$ obtained by right multiplication by $a=s_0$, $b=s_2$, or $c=s_2s_1s_2$ are shown in Figure \ref{fig:affine-a2-counterexample}. Therefore, $\overline{C}$ is not biclosed, hence cannot be the join of $E$ and $F$ in $\BC(\pp)$.
\end{example}

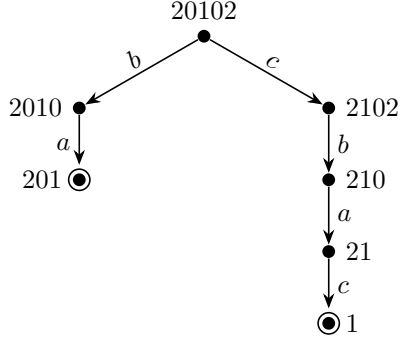
\begin{figure}[htbp]
\centering
\begin{tikzpicture}[x=1cm,y=0.95cm,>=Stealth,
  vertex/.style={circle,fill=black,inner sep=1.7pt},
  every edge/.style={draw,->,semithick},
  every label/.style={font=\small},
  lab/.style={font=\small,inner sep=1pt}]

  \node[vertex,label=above:{$20102$}]
    (top) at (0,4) {};
  \node[vertex,label=left:{$2010$}]
    (left4) at (-1.65,3) {};
  \node[vertex,label=left:{$201$}]
    (left3) at (-1.65,2) {};
  \node[vertex,label=right:{$2102$}]
    (right4) at (1.65,3) {};
  \node[vertex,label=right:{$210$}]
    (right3) at (1.65,2) {};
  \node[vertex,label=right:{$21$}]
    (right2) at (1.65,1) {};
  \node[vertex,label=right:{$1$}]
    (right1) at (1.65,0) {};

  \draw[semithick] (left3) circle[radius=4pt];
  \draw[semithick] (right1) circle[radius=4pt];

  \path
    (top) edge node[lab,sloped,above=1pt] {$b$} (left4)
    (top) edge node[lab,sloped,above=1pt] {$c$} (right4)
    (left4) edge[shorten >=3pt]
      node[lab,left=2pt] {$a$} (left3)
    (right4) edge node[lab,right=2pt] {$b$} (right3)
    (right3) edge node[lab,right=2pt] {$a$} (right2)
    (right2) edge[shorten >=3pt]
      node[lab,right=2pt] {$c$} (right1);

\end{tikzpicture}
\caption{All decreasing paths from $20102$ using the right reflection
labels $a,b,c$. Neither circled endpoint is $e$.}
\label{fig:affine-a2-counterexample}
\end{figure}

Biagioli and Perrone~\cite{biagioliConjectureDyerJoin2026} proved the conjecture for finite dihedral groups and groups of type $A$, and verified it computationally in types $H_3$ and $F_4$. Dermenjian~\cite{dermenjianBruhatPreclosure2025} showed that inversion sets, and more generally initial sections of reflection orders, are fixed by Bruhat preclosure. He also proved that iterated Bruhat preclosure yields the inversion set of a bounded weak join in arbitrary Coxeter groups, and established idempotence on arbitrary reflection subsets in type $A$, giving another proof in that type. His examples show that idempotence fails for arbitrary reflection subsets in other types. Subsequently, Biagioli and Perrone~\cite{biagioliComputingJoinsWeak2026} developed an algorithm for joins of signed permutations, extending Markowsky's algorithm, and announced a proof of the conjecture in type $B$ to appear in a forthcoming article. More recently, Lim~\cite{limInversionSetsJoins2026} proved the conjecture for finite simply-laced Coxeter groups using palindromic Bruhat preclosure. Together with the previously known cases, this establishes the conjecture for all finite Coxeter groups.

\section{Preliminaries}\label{sec:prelim}
\subsection{Coxeter groups, roots, and orders}
Let $(W,S)$ be a Coxeter system, with identity $e$, length function $\ell$, and reflection set
$T=\{wsw^{-1}:w\in W,\ s\in S\}$.
We use the standard geometric representation on the real vector space $V$ with basis $\{\alpha_s:s\in S\}$, and write $\Phi=\pp\sqcup\Phi^-$ for its root system, where $\Phi^-=-\pp$.
For $\alpha\in\Phi$, let $t_\alpha$ be the corresponding reflection. Thus
$t_{-\alpha}=t_\alpha$ and $wt_\alpha w^{-1}=t_{w\alpha}$.
Our inversion-set conventions are
\[
\begin{aligned}
 I(w)&=\{\alpha\in\pp:w^{-1}\alpha\in\Phi^-\},\\
 I_R(w)&=\{\alpha\in\pp:w\alpha\in\Phi^-\}=I(w^{-1}),
 &D(w)&=\pp\setminus I(w).
\end{aligned}
\]
In particular, $|I(w)|=|I_R(w)|=\ell(w)$ and $I(w)=-wI_R(w)$.
For $\alpha\in\pp$, the root-sign criteria give
\begin{equation}\label{eq:root-sign}
\begin{aligned}
 \ell(wt_\alpha)>\ell(w)&\iff w\alpha\in\pp,\\
 \ell(t_\alpha w)>\ell(w)&\iff w^{-1}\alpha\in\pp.
\end{aligned}
\end{equation}
See \cite[Chapter~4]{bjornerCombinatoricsCoxeterGroups2005} for these facts and the geometric representation.

We write $\wedge_R$ for the meet in right weak order and $\covL$ for a cover in left weak order. A factorization $w=uv$ is \emph{reduced} if $\ell(w)=\ell(u)+\ell(v)$. For a reduced factorization $w=uv$,
\begin{equation}\label{eq:inversion-product}
 I(w)=I(u)\sqcup uI(v),\qquad
 I_R(w)=I_R(v)\sqcup v^{-1}I_R(u).
\end{equation}

Bruhat order $\le_B$ is the transitive closure of the relations $x<_B xt$ with $t\in T$ and $\ell(x)<\ell(xt)$. Such a relation is a cover, denoted by $x\covB xt$, precisely when $\ell(xt)=\ell(x)+1$.
We use the strong exchange property, the subword, chain, and lifting properties of Bruhat order, and the connectivity of reduced expressions under braid moves
\cite[Theorems~1.4.3, 2.2.2, 2.2.6, 3.3.1 and Proposition~2.2.7]{bjornerCombinatoricsCoxeterGroups2005}.
Write $D_L(z)=\{s\in S:\ell(sz)<\ell(z)\}$ and $D_R(z)=\{s\in S:\ell(zs)<\ell(z)\}$.

\subsection{Parabolic subgroups}
For $H\subseteq S$, let $W_H=\langle H\rangle$, and set $V_H=\operatorname{span}_{\R}\{\alpha_s:s\in H\}$, $\Phi_H=W_H\{\alpha_s:s\in H\}$, and $\Phi_H^+=\Phi_H\cap\pp$.
Then $\Phi_H=\Phi\cap V_H$. Length, weak order, and Bruhat order in $W_H$ are the restrictions of their counterparts in $W$, and $I(w)\subseteq\Phi_H^+$ for $w\in W_H$.
We will use the characterizations
\begin{equation}\label{eq:parabolic-minimal}
\begin{aligned}
 D_L(x)\cap H=\varnothing
 &\iff I(x)\cap\Phi_H^+=\varnothing\\
 &\iff \ell(vx)=\ell(v)+\ell(x)\quad\text{for every }v\in W_H
\end{aligned}
\end{equation}
of the minimal-length representative in $W_Hx$; see \cite[\S2.4 and Chapter~4]{bjornerCombinatoricsCoxeterGroups2005}.

Suppose $H=\{s,t\}$ with $s\ne t$ and $m_{st}<\infty$, and let $w_H$ be the longest element of $W_H$.
If $w_Hx$ is reduced, the two alternating reduced expressions for $w_H$ give two suffix chains from $x$ to $w_Hx$.
The roots added to $I_R$ along these chains are the roots in $x^{-1}\Phi_H^+$, in opposite orders.
The endpoint roots are $x^{-1}\alpha_s$ and $x^{-1}\alpha_t$; each other root is a linear combination of these two with both coefficients strictly positive.
This is the usual dihedral root sequence, transported by $x^{-1}$; compare \cite[\S3]{dyerWeakOrderCoxeter2019}.

\subsection{Root closures and reachability}
For $A\subseteq\pp$, let $\cone(A)$ be its nonnegative real span, using finite linear combinations, and put $d(A)=\cone(A)\cap\pp$.
We use the definitions of $\ccl$, coclosedness, biclosedness, and boundedness from Section~\ref{sec:intro}. All complements are taken in $\pp$, and all unqualified references to closedness mean $2$-closedness.

The sets $I(w)$ and $D(w)$ are closed under both $\ccl$ and $d$.
A finite subset of $\pp$ is biclosed if and only if it is an inversion set
\cite[Lemma~4.1(iv)]{dyerWeakOrderCoxeter2019}; see also \cite[Proposition~2.11]{hohlwegInversionSetsWeak2016} for the equivalent separation characterization in finite rank.
Consequently, if $f\in V^*$ is nonzero on every root and $B_f=\{\alpha\in\pp:f(\alpha)<0\}$ is finite, then $B_f=I(v)$ for some $v\in W$: both $B_f$ and its complement are conically closed, so $B_f$ is biclosed.

Recall that $\RR(C)$ consists of the endpoints of finite paths
\[
 e=x_0\longrightarrow x_1\longrightarrow\cdots\longrightarrow x_m,
 \qquad x_i=x_{i-1}t_{\alpha_i},\quad
 \alpha_i\in C,\quad \ell(x_i)>\ell(x_{i-1}).
\]
The path of length zero is allowed, so $e\in\RR(C)$.
The Bruhat preclosure $\overline C$ consists of roots whose corresponding reflections are endpoints of such paths. This agrees with the operation studied in~\cite{dermenjianBruhatPreclosure2025} under the correspondence between roots and reflections.

\begin{example}\label{ex:affine-a2-join}
Let $W$ be a Coxeter group of type $\widetilde A_2$, with simple reflections $s_0,s_1,s_2$,
simple roots $\alpha_0,\alpha_1,\alpha_2$, and
$\delta=\alpha_0+\alpha_1+\alpha_2$. Take
\[
 C=I(s_0s_1)\cup I(s_0s_2)
  =\{\alpha_0,\,\alpha_0+\alpha_1,\,\alpha_0+\alpha_2\}.
\]
The right weak join is $w=s_0s_1s_2s_1=s_0s_2s_1s_2$, and
\[
 \overline C=\ccl(C)=I(w)=C\sqcup\{\alpha_0+\delta\}.
\]
Figure~\ref{fig:affine-a2-join} shows the reachable graph, with arrows given
by right multiplication by $a=s_0$, $b=s_0s_1s_0$, or $c=s_0s_2s_0$.
We abbreviate $i_1\cdots i_k=s_{i_1}\cdots s_{i_k}$.
For instance, the path
\[
 e\xrightarrow{b}010\xrightarrow{c}0120\xrightarrow{b}01210
   =t_{\alpha_0+\delta}
\]
reaches the reflection corresponding to the new root in $\overline C$.
The twelve vertices form $\RR(C)=[e,w]_B w^{-1}$.
\end{example}

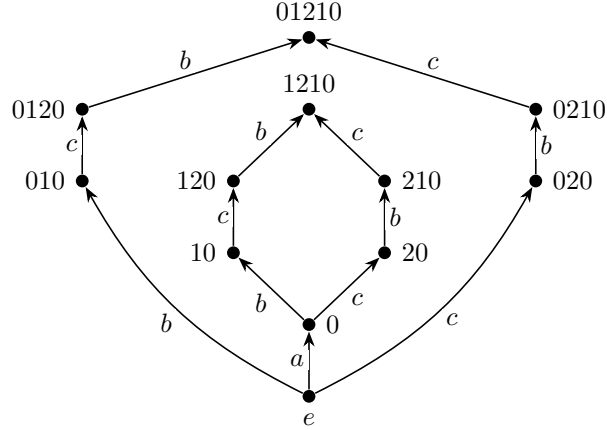
\begin{figure}[htbp]
\centering
\begin{tikzpicture}[x=1cm,y=0.95cm,>=Stealth,
  vertex/.style={circle,fill=black,inner sep=1.7pt},
  every edge/.style={draw,->,semithick},
  every label/.style={font=\small},
  lab/.style={font=\small,fill=white,inner sep=1.5pt}]

  \node[vertex,label=below:$e$] (e) at (0,0) {};
  \node[vertex,label=right:$0$] (0) at (0,1) {};
  \node[vertex,label=left:$10$] (10) at (-1,2) {};
  \node[vertex,label=right:$20$] (20) at (1,2) {};
  \node[vertex,label=left:$120$] (120) at (-1,3) {};
  \node[vertex,label=right:$210$] (210) at (1,3) {};
  \node[vertex,label=above:$1210$] (1210) at (0,4) {};
  \node[vertex,label=left:$010$] (010) at (-3,3) {};
  \node[vertex,label=right:$020$] (020) at (3,3) {};
  \node[vertex,label=left:$0120$] (0120) at (-3,4) {};
  \node[vertex,label=right:$0210$] (0210) at (3,4) {};
  \node[vertex,label=above:$01210$] (01210) at (0,5) {};

  \path
    (e) edge node[lab,left] {$a$} (0)
    (e) edge[bend left=17] node[lab,below left] {$b$} (010)
    (e) edge[bend right=17] node[lab,below right] {$c$} (020)
    (0) edge node[lab,below left] {$b$} (10)
    (0) edge node[lab,below right] {$c$} (20)
    (10) edge node[lab,left] {$c$} (120)
    (20) edge node[lab,right] {$b$} (210)
    (120) edge node[lab,above left] {$b$} (1210)
    (210) edge node[lab,above right] {$c$} (1210)
    (010) edge node[lab,left] {$c$} (0120)
    (020) edge node[lab,right] {$b$} (0210)
    (0120) edge node[lab,above left] {$b$} (01210)
    (0210) edge node[lab,above right] {$c$} (01210);

\end{tikzpicture}
\caption{The complete reachable graph for
$C=I(s_0s_1)\cup I(s_0s_2)$ in type $\widetilde A_2$.
Each arrow is labeled by the reflection acting on the right.}
\label{fig:affine-a2-join}
\end{figure}
\FloatBarrier

\section{Proof of the main theorem}\label{sec:proof}
\subsection{The inversion set generated by a bounded coclosed set}\label{sec:closure}
We first reduce to finite rank. Let $K\subseteq S$ be the finite support of a reduced expression for the element $x$ in Theorem~\ref{thm:main}. Then $C\subseteq I(x)\subseteq\Phi_K^+$.
Since $\Phi\cap V_K=\Phi_K$, the ambient and parabolic $2$-closures of $C$ agree. Moreover, $\Phi_K^+\setminus C$ is closed in $\Phi_K^+$, so $C$ remains coclosed in the parabolic root system. Every allowed reflection lies in $W_K$, and the length functions agree, so the reachable sets agree as well. Finally, the weak and Bruhat lower intervals of an element of $W_K$ lie in $W_K$. It is therefore enough to prove the theorem in finite rank. From now on, we assume $S$ is finite.

For $u\in W$, put $E_\pm(u)=\{\alpha\in\pp:\ell(t_\alpha u)=\ell(u)\pm1\}$.
Dyer's generation criteria \cite[Lemma~1.7 and Proposition~11.6]{dyerWeakOrderCoxeter2019} give, for $A\subseteq\pp$,
\begin{align}
 \ccl(A)=I(u)&\iff E_-(u)\subseteq A\subseteq I(u),\label{eq:minus}\\
 \ccl(A)=D(u)&\iff d(A)=D(u)
 \iff E_+(u)\subseteq A\subseteq D(u).\label{eq:plus}
\end{align}
If $\alpha\in E_-(u)$, Dyer's Theorem~1.8(ii) further gives
\begin{equation}\label{eq:dyer-cover}
\begin{aligned}
 v=u\wedge_R t_\alpha u
 &=\max\{z\le_R u:\alpha\notin I(z)\},\\
 D(v)&=\ccl\bigl(D(u)\cup\{\alpha\}\bigr).
\end{aligned}
\end{equation}

\begin{prop}\label{prop:closure}
Under the hypotheses of Theorem~\ref{thm:main}, there is a unique $w\in W$ such that $\ccl(C)=I(w)$.
Moreover, $w\le_R u$ whenever $C\subseteq I(u)$.
\end{prop}
\begin{proof}
Choose $w$ of minimum length subject to $C\subseteq I(w)$; such an element exists because $C\subseteq I(x)$.
If $\ccl(C)\ne I(w)$, then~\eqref{eq:minus} gives $\alpha\in E_-(w)\setminus C$.
The closed set $\pp\setminus C$ contains $D(w)\cup\{\alpha\}$, hence contains $D(v)$ for $v=w\wedge_R t_\alpha w$, by~\eqref{eq:dyer-cover}.
Thus $C\subseteq I(v)$. Since $v\le_R t_\alpha w$, we have $\ell(v)\le\ell(w)-1$, contradicting the choice of $w$.
Therefore $\ccl(C)=I(w)$. Uniqueness follows because inversion sets determine their elements.
If $C\subseteq I(u)$, closedness of $I(u)$ gives $I(w)=\ccl(C)\subseteq I(u)$, and hence $w\le_R u$.
\end{proof}

\subsection{Roots forced by least carriers}\label{sec:mandatory}
Fix the element $w$ from Proposition~\ref{prop:closure}.
For $u\in W$ and $\alpha,\beta\in\pp$, define
\[
 \A_u(\alpha)=\{h\le_R u:\alpha\notin I(h)\},\qquad
 \C_u(\beta)=\{d\le_L u:\beta\in I_R(d)\}.
\]
We call the elements of $\A_u(\alpha)$ \emph{anti-carriers} of $\alpha$ below $u$, and the elements of $\C_u(\beta)$ \emph{carriers} of $\beta$ below $u$.
These sets are, respectively, a lower ideal of $[e,u]_R$ and an upper ideal of $[e,u]_L$.

If $\beta\in I_R(u)$ and $\alpha=-u\beta$, then the anti-isomorphism from $[e, u]_R$ to $[e, u]_L$ given by $h\mapsto h^{-1}u=d$, sending a prefix of $u$ to the complementary suffix of $u$, identifies $\A_u(\alpha)$ and $\C_u(\beta)$: $\alpha\notin I(h)$ means $h^{-1}\alpha=-d\beta\in\pp$, which is equivalent to $\beta\in I_R(d)$.
This is the anti-isomorphism of \cite[Proposition~2.19]{readingSortableElementsInfinite2010} followed by the map $g\mapsto g^{-1}$.

\begin{example}[Carriers in type $A_2$]
Let $W=\langle s,t\mid s^2=t^2=(st)^3=e\rangle$ and $w_0=sts=tst$.
We have $I_R(st)=\{\alpha_t,\alpha_s+\alpha_t\}$ and $I_R(ts)=\{\alpha_s,\alpha_s+\alpha_t\}$.
Thus $\C_{w_0}(\alpha_s)=\{s,ts,w_0\}$ has least element $s$, whereas $\C_{w_0}(\alpha_s+\alpha_t)=\{st,ts,w_0\}$ has two incomparable minimal elements $st$ and $ts$, and hence no least element; see Figure~\ref{fig:carriers-a2}.
\end{example}

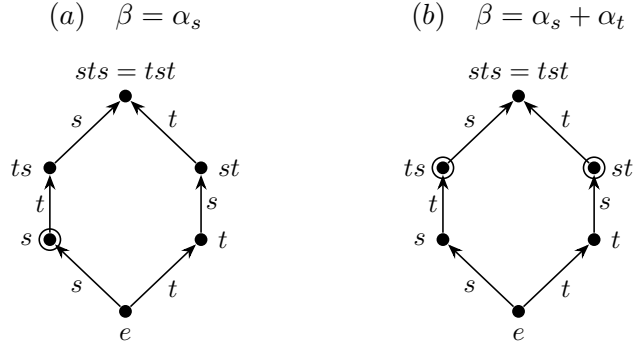
\begin{figure}[htbp]
\centering
\begin{tikzpicture}[x=1cm,y=0.95cm,>=Stealth,
  vertex/.style={circle,fill=black,inner sep=1.7pt},
  every edge/.style={draw,->,semithick},
  every label/.style={font=\small},
  lab/.style={font=\small,fill=white,inner sep=1.5pt}]

  \begin{scope}
    \node at (0,4.1) {$(a)\quad\beta=\alpha_s$};
    \node[vertex,label=below:$e$] (e) at (0,0) {};
    \node[vertex,label=left:$s$] (s) at (-1,1) {};
    \node[vertex,label=right:$t$] (t) at (1,1) {};
    \node[vertex,label=left:$ts$] (ts) at (-1,2) {};
    \node[vertex,label=right:$st$] (st) at (1,2) {};
    \node[vertex,label={above:{$sts=tst$}}] (w) at (0,3) {};

    \draw[semithick] (s) circle[radius=4pt];

    \path
      (e) edge node[lab,below left] {$s$} (s)
      (e) edge node[lab,below right] {$t$} (t)
      (s) edge node[lab,left] {$t$} (ts)
      (t) edge node[lab,right] {$s$} (st)
      (ts) edge node[lab,above left] {$s$} (w)
      (st) edge node[lab,above right] {$t$} (w);
  \end{scope}

  \begin{scope}[xshift=5.2cm]
    \node at (0,4.1) {$(b)\quad\beta=\alpha_s+\alpha_t$};
    \node[vertex,label=below:$e$] (e) at (0,0) {};
    \node[vertex,label=left:$s$] (s) at (-1,1) {};
    \node[vertex,label=right:$t$] (t) at (1,1) {};
    \node[vertex,label=left:$ts$] (ts) at (-1,2) {};
    \node[vertex,label=right:$st$] (st) at (1,2) {};
    \node[vertex,label={above:{$sts=tst$}}] (w) at (0,3) {};

    \draw[semithick] (ts) circle[radius=4pt];
    \draw[semithick] (st) circle[radius=4pt];

    \path
      (e) edge node[lab,below left] {$s$} (s)
      (e) edge node[lab,below right] {$t$} (t)
      (s) edge node[lab,left] {$t$} (ts)
      (t) edge node[lab,right] {$s$} (st)
      (ts) edge node[lab,above left] {$s$} (w)
      (st) edge node[lab,above right] {$t$} (w);
  \end{scope}

\end{tikzpicture}
\caption{Carriers in the left weak order of type $A_2$,
with $w_0=sts=tst$.
The carrier sets are $\{s,ts,w_0\}$ in (a) and
$\{st,ts,w_0\}$ in (b).
Circled vertices are minimal carriers.
Each arrow is labeled by the simple reflection acting on the left.}
\label{fig:carriers-a2}
\end{figure}

\begin{lemma}\label{lem:mandatory}
Let $\alpha\in I(w)$. If $\A_w(\alpha)$ has a greatest element $v$, then
\begin{equation}\label{eq:avoidance}
 \ccl\bigl(D(w)\cup\{\alpha\}\bigr)=D(v),\qquad \alpha\in C.
\end{equation}
Consequently, if $\beta\in I_R(w)$ and $\C_w(\beta)$ has a least element, then $-w\beta\in C$.
\end{lemma}
\begin{proof}
Set $\Gamma=D(w)\cup\{\alpha\}$ and $K=\cone(\Gamma)$.
Since $v\le_R w$ and $\alpha\notin I(v)$, we have $\Gamma\subseteq D(v)$.
By \cite[Lemma~11.5]{dyerWeakOrderCoxeter2019}, $\cone(D(w))$ is polyhedral. Thus $K$, obtained by adjoining one generator, is a finitely generated closed cone.

Take $\gamma\in\pp\setminus K$. Separation of a point from a closed convex cone \cite[\S11]{rockafellarConvexAnalysis2015} gives $f\in V^*$ such that $f(\delta)\ge0$ for all $\delta\in K$ and $f(\gamma)<0$.
Let $h\in V^*$ satisfy $h(\alpha_s)=1$ for every $s\in S$, so $h$ is strictly positive on $\pp$.
For sufficiently small $\eps>0$, the functional $f'=f+\eps h$ remains negative on $\gamma$ and is strictly positive on $\Gamma$, in particular on $D(w)$.
Only the finitely many positive roots in $I(w)$ remain to be checked for zeros. By avoiding the finitely many corresponding values of $\eps$, we may assume that $f'$ is nonzero on every root.
The set $B=\{\delta\in\pp:f'(\delta)<0\}$ is biclosed and contained in the finite set $I(w)$. Hence $B=I(z)$ for some $z\in W$.
We have $z\le_R w$ and $\alpha\notin I(z)$, so $z\le_R v$ by the defining property of $v$.
In particular, $\gamma\in I(z)\subseteq I(v)$.

It follows that $\pp\setminus K\subseteq I(v)$, or equivalently $D(v)\subseteq K\cap\pp=d(\Gamma)$.
The reverse inclusion follows from $\Gamma\subseteq D(v)$ and conic closedness of $D(v)$.
Thus $d(\Gamma)=D(v)$, and~\eqref{eq:plus} gives $\ccl(\Gamma)=D(v)$.

If $\alpha\notin C$, the closed set $\pp\setminus C$ contains $\Gamma$, hence $D(v)$.
Then $C\subseteq I(v)$, and consequently $I(w)=\ccl(C)\subseteq I(v)$, contradicting $\alpha\in I(w)\setminus I(v)$.
The assertion about carriers follows from the order-reversing bijection above.
\end{proof}
\begin{remark}
In the case $C=I(u)\cup I(v)$, the assertion $\alpha\in C$ in Lemma~\ref{lem:mandatory} has a simpler proof. If $\alpha\notin C$, then $u,v\in\A_w(\alpha)$. Writing $z=\max\A_w(\alpha)$, we have $u,v\le_Rz$, and then $w=u\vee_Rv\le_Rz$, contradicting $\alpha\in I(w)\setminus I(z)$.
\end{remark}

Three results about carriers will be used later.
First, if $\beta\in\pp$ and $zt_\beta\covB z$, then applying~\eqref{eq:dyer-cover} to $\alpha=-z\beta$ gives
\begin{equation}\label{eq:carrier-cover}
 \min\C_z(\beta)=(z\wedge_R zt_\beta)^{-1}z.
\end{equation}
Second, least carriers restrict to smaller suffix intervals:
\begin{equation}\label{eq:carrier-inherit}
 p=\min\C_z(\beta),\quad b\in\C_z(\beta)
 \quad\Longrightarrow\quad p=\min\C_b(\beta).
\end{equation}
Indeed, $p\le_L b$, and $\C_b(\beta)\subseteq\C_z(\beta)$. Finally,
\begin{equation}\label{eq:minicover}
p=\min\C_z(\beta)
\quad\Longrightarrow\quad
pt_\beta\lessdot_L p.
\end{equation}
Indeed, since $\beta\in I_R(p)$, we have $p\ne e$. Write $p=sq$ as a reduced factorization with $s\in S$. Since $q<_Lp\le_Lz$, minimality gives $\beta\notin I_R(q)$. Thus $\beta\in I_R(p)\setminus I_R(q)=\{q^{-1}\alpha_s\}$, and hence $pt_\beta=q\lessdot_Lp$. Note that $pt_\beta$ is in fact the only element that $p$ covers in left weak order.

\subsection{Reduction to a root-selection problem}\label{sec:reduct}
We now turn to the main part of the proof. In Section~\ref{sec:closure}, we showed that $\ccl(C)=I(w)$, so $\RR(C)\subseteq\RR(I(w))$. To prove the main theorem, it remains to show that $\RR(I(w))\subseteq[e,w]_Bw^{-1}\subseteq\RR(C)$. We begin with the first inclusion.
\begin{lemma}\label{lem:IWsubsetew}
Let $(W,S)$ be any Coxeter system and $w\in W$. Then $\RR(I(w))\subseteq[e,w]_Bw^{-1}$.
\end{lemma}
\begin{proof}
We proceed by induction on the length of $u\in\RR(I(w))$. If $u=e$, then $e=w\cdot w^{-1}\in[e,w]_Bw^{-1}$. Suppose that the result is true for all $u\in \RR(I(w))$ with $\ell(u)\leq k$. Given $u\in \RR(I(w))$ with $\ell(u)=k+1$, write $u=vt_{\alpha}$, where $\ell(v)\le k$, $v\in\RR(I(w))$ and $\alpha\in I(w)$. By the induction hypothesis, there exists \(x\leq_B w\) such that $v=xw^{-1}$. Hence $u=vt_\alpha=xw^{-1}t_\alpha=xt_{w^{-1}\alpha}w^{-1}$. Since $\ell(vt_\alpha)>\ell(v)$ and $\alpha\in\Phi^+$, $v\alpha\in\Phi^+$. On the other hand, $\alpha\in I(w)$ implies $w^{-1}\alpha\in\Phi^-$, so we set $\beta=-w^{-1}\alpha\in\Phi^+$. Then $t_\beta=t_{w^{-1}\alpha}$ and $x\beta=-xw^{-1}\alpha=-v\alpha\in\Phi^-$. Combining the results above, we obtain $xt_\beta<_B x\leq_B w$. Therefore, $u=xt_\beta w^{-1}\in[e,w]_B w^{-1}$.
\end{proof}

The second inclusion is more complicated, we use the following reduction.
\begin{lemma}\label{lem:easyreduct}
Suppose that for every $y<_B w$, there is $\beta\in I_R(w)$ such that $y<_B yt_\beta\leq_Bw$ and $-w\beta\in C$. Then $[e,w]_Bw^{-1}\subseteq \RR(C)$.
\end{lemma}
\begin{proof}
We proceed by descending induction on the length of $y\in[e,w]_B$. If $y=w$, then $ww^{-1}=e\in\RR(C)$. Now suppose $y<_Bw$. By assumption, there is $\beta\in I_R(w)$ such that $y<_B yt_\beta\leq_Bw$ and $-w\beta\in C$. Then $\beta\notin I_R(y)$, so $y\beta\in \Phi^+$. By the induction hypothesis, $yt_\beta w^{-1}\in \RR(C)$. Moreover, $yw^{-1}=yt_\beta w^{-1}t_{w\beta}$. Since $\beta\in I_R(w)$, we have $w\beta\in \Phi^-$. Together with $yt_\beta w^{-1}(w\beta)=-y\beta\in\Phi^-$, this implies $yw^{-1}>_Byt_\beta w^{-1}$. Since $-w\beta\in C$ and $yt_{\beta}w^{-1}\in\RR(C)$, we conclude that $yw^{-1}\in\RR(C)$.
\end{proof}

Combining this with Lemma~\ref{lem:mandatory} gives the following corollary.
\begin{cor}\label{cor:rootreduct}
Suppose that for every $y<_B w$, there is $\beta\in I_R(w)$ such that $y<_B yt_\beta\leq_Bw$ and $\C_w(\beta)$ has a least element. Then $[e,w]_Bw^{-1}\subseteq \RR(C)$.
\end{cor}

\subsection{Selection of the root}\label{sec:selection}
We seek a root $\beta$ satisfying three conditions: $\beta\in I_R(w)\backslash I_R(y)$, $yt_\beta\leq_Bw$, and $\C_w(\beta)$ has a least element. The first condition alone is easy to satisfy, since $I_R(w)$ is not a subset of $I_R(y)$. For a simple root $\beta$ satisfying the first condition, the lifting property \cite[Proposition~2.2.7]{bjornerCombinatoricsCoxeterGroups2005} gives the second. Although $\beta$ need not be simple, we will reduce to a situation where the lifting property applies. The main difficulty is therefore the third condition.

Write $I_R(w,y)=I_R(w)\backslash I_R(y)$. For every $\beta\in I_R(w,y)$ and every reduced expression $w=s_1\cdots s_k$, there is exactly one $j\in\{1,\ldots,k\}$ such that $w_j=w_{j+1}t_{\beta}$, where $w_j=s_j\cdots s_k$ and $w_{k+1}=e$. Applying~\eqref{eq:carrier-cover}, we find that $\C_{w_j}(\beta)$ has a least element. This alone does not imply that $\C_w(\beta)$ has a least element. Our strategy is to find $\beta\in I_R(w,y)$ such that the least element of $\C_{w_j}(\beta)$ is the same for all reduced expressions of $w$. We therefore examine what happens under a single braid move.

\begin{lemma}\label{lem:braid}
Let $H=\{s,t\}\subseteq S$, where $s\ne t$ and $m_{st}<\infty$, and suppose that $w_Hx$ is reduced.
Consider an internal edge $u\covL v=su$ on one of the two suffix chains from $x$ to $w_Hx$, where internal means $u>_L x$ and $v<_L w_Hx$. Let $\beta=u^{-1}\alpha_s$ and $\eta_r=x^{-1}\alpha_r$ for $r\in H$. Then the least carriers $p=\min\C_v(\beta)$ and $q_r=\min\C_{rx}(\eta_r)$ exist, and $\ell(q_r)<\ell(p)$ for both $r\in H$.
\end{lemma}
\begin{proof}
Since $vt_\beta=u\covB v$ and $rxt_{\eta_r}=x\covB rx$, formula~\eqref{eq:carrier-cover} gives $a=u\wedge_R v$, $b_r=x\wedge_R rx$, $p=a^{-1}v$, and $q_r=b_r^{-1}rx$.
The nonidentity, nonlongest dihedral factors of $u$ and $v$ have distinct singleton left descent sets in $H$. Since $w_Hx$ is reduced, the same holds for $D_L(u)\cap H$ and $D_L(v)\cap H$.
A common prefix $a$ of $u$ and $v$ has no left descent in $H$. By~\eqref{eq:parabolic-minimal}, $I(a)\cap\Phi_H^+=\varnothing$.

We use the following consequence of the strong exchange property \cite[Theorem~1.4.3]{bjornerCombinatoricsCoxeterGroups2005}: if $a\le_R z$, $\gamma\in D(a)$, and $t_\gamma z\covB z$, then $a\le_R t_\gamma z$.
Indeed, in a reduced expression $z=ab$, the strong exchange property deletes one letter. Deletion in the prefix $a$ would imply $\gamma\in I(a)$, so the deleted letter lies in $b$. The remaining expression is reduced because the length drops by exactly one.

Write $v=v_Hx$ with $v_H\in W_H$. Its length in $W_H$ is at least two, so $e,s,t\le_B v_H$.
Choose saturated descending Bruhat chains in $W_H$ from $v_H$ to $e$, to $s$, and to $t$, and append $x$ to every vertex.
The resulting steps remain covers and their left reflection labels are roots in $\Phi_H^+$.
The preceding exchange argument therefore applies at every step and yields $a\le_R x, sx, tx$.
Hence $a\le_R b_r$, and
\[
 \ell(q_r)=\ell(rx)-\ell(b_r)
 \le\ell(x)+1-\ell(a)
 <\ell(v)-\ell(a)=\ell(p).\qedhere
\]
\end{proof}

\begin{prop}\label{prop:rootselect}
For every $y,w\in W$ with $y<_B w$, there is $\beta\in I_R(w)$ such that $y\covB yt_\beta\le_B w$ and $\C_w(\beta)$ has a least element.
\end{prop}\begin{proof}
Define $\mathscr P=\{(c,\beta):c\le_Lw,\ \beta\in\Phi^+\setminus I_R(y),\ ct_\beta\lessdot_Lc\}$.
This set is nonempty: choose $\beta\in I_R(w)\setminus I_R(y)$ and follow the suffix chain of any reduced expression of $w$ to the step where $\beta$ enters the right inversion set. The upper element of that step, together with $\beta$, belongs to $\mathscr P$. Choose $(c,\beta)\in\mathscr P$ with $\ell(c)$ minimal. We will show that $c$ is the least element of $\C_w(\beta)$.

For a reduced expression $\s=(s_1,\ldots,s_m)$ of $w$, put $w_i=s_i\cdots s_m$ and $w_{m+1}=e$. There is a unique index $j$ with $\beta\in I_R(w_j)\setminus I_R(w_{j+1})$, and then $w_j=w_{j+1}t_\beta$.
Define $P_w(\s)=\min\C_{w_j}(\beta)$, which exists by~\eqref{eq:carrier-cover}. By~\eqref{eq:minicover}, we have $(P_w(\s),\beta)\in\mathscr P$, and therefore $\ell(P_w(\s))\ge\ell(c)$.
Write $c=s'd$ as a reduced factorization with $s'\in S$ and $d=ct_\beta$, and extend it on the left to a reduced expression $\s$ of $w$. The step where $\beta$ enters is then $d<_Lc$, so $P_w(\s)\le_Lc$. The preceding inequality gives $P_w(\s)=c$.

We next show that this equality is preserved by a braid move. Suppose two reduced expressions $\s_1,\s_2$ connected by a braid move differ in positions $k+1,\ldots,l$, where they form two alternating words representing $w_H$ for $H=\{s,t\}$. Set $x=w_{l+1}$ and $z=w_Hx$, and assume $P_w(\s_1)=c$. The roots contributed by this block are $x^{-1}\Phi_H^+$, listed in opposite orders in the two expressions.

If $j<k+1$ or $j>l$, the corresponding suffix and its predecessor are unchanged. Hence $P_w(\s_2)=c$.

If $j=k+1$ or $j=l$, then $\beta=x^{-1}\alpha_r$ for some $r\in H$. Its two occurrences have upper suffixes $rx$ and $z$. As $z>_L rx$, applying \eqref{eq:carrier-inherit} to $rx\in\C_z(\beta)$ gives $P_w(\s_1)=P_w(\s_2)=c$.

Finally, suppose $k+1<j<l$. After relabeling $s$ and $t$ if necessary, put $u=w_{j+1}$ and $v=w_j=su$. Then Lemma~\ref{lem:braid} gives $\ell(q_r)<\ell(P_w(\s_1))=\ell(c)$ for $q_r=\min\C_{rx}(x^{-1}\alpha_r)$ and $r\in H$.
Each $q_r$ is a suffix of $w$ and satisfies $q_rt_{x^{-1}\alpha_r}\lessdot_Lq_r$ by~\eqref{eq:minicover}. Minimality of $\ell(c)$ therefore forces $x^{-1}\alpha_s,x^{-1}\alpha_t\in I_R(y)$. But the interior root $\beta$ is a positive linear combination of these two roots. Since $I_R(y)$ is $2$-closed, this contradicts $\beta\notin I_R(y)$. Thus the interior case cannot occur.

Connectivity under braid moves now gives $P_w(\s)=c$ for every reduced expression of $w$. To deduce that $c$ is the least carrier below $w$, take any $b\in\C_w(\beta)$ and extend a reduced expression of $b$ on the left to one of $w$. Since $\beta\in I_R(b)$, its entry step occurs within the suffix corresponding to $b$. Its upper suffix $w_j$ satisfies $w_j\le_Lb$, and hence $c=P_w(\s)\le_Lw_j\le_Lb$.
As $c\in\C_w(\beta)$, this proves $c=\min\C_w(\beta)$.

It remains to prove $yt_\beta\le_Bw$. We claim that $I_R(d)\subseteq I_R(y)$. Otherwise, choose $\gamma\in I_R(d)\setminus I_R(y)$ and consider its entry step along a reduced expression of $d$. The upper suffix $b\le_Ld$ gives $(b,\gamma)\in\mathscr P$ with $\ell(b)\le\ell(d)<\ell(c)$, a contradiction. Thus $d\le_Ly$.

Write $y=y'd$ and $w=w'd$ reduced. Canceling the common final letters of a reduced expression of $d$ by the lifting property gives $y'\le_Bw'$. Since $c=s'd\le_Lw$, $s'$ is a right descent of $w'$. On the other hand, as $c=s'd\nleq_L y$, $s'$ is not a right descent of $y'$. The right-handed lifting property, obtained by inversion from \cite[Proposition~2.2.7]{bjornerCombinatoricsCoxeterGroups2005}, yields $y's'\le_Bw'$. Moreover, $\beta\notin I_R(y)$ gives $\ell(yt_\beta)>\ell(y)$, while $yt_\beta=y's'd$ has an expression of length $\ell(y)+1$. This expression is therefore reduced. By the subword property, a reduced expression of $y's'$ occurs as a subword of a reduced expression of $w'$. Appending a reduced expression of $d$ to both gives $yt_\beta=y'dt_\beta=y's'd\le_Bw$. Since $\ell(yt_\beta)=\ell(y)+1$, this proves $y\lessdot_B yt_\beta\le_Bw$ and completes the proof.
\end{proof}

\subsection{Completion of the proof}
\begin{proof}[Proof of Theorem~\ref{thm:main}]
By Proposition~\ref{prop:closure}, there is a unique $w$ with $\ccl(C)=I(w)$, and $w\le_Rx$. Proposition~\ref{prop:rootselect} and Corollary~\ref{cor:rootreduct} give $[e,w]_Bw^{-1}\subseteq\RR(C)$. Together with Lemma~\ref{lem:IWsubsetew} and $C\subseteq I(w)$, this yields $\RR(C)\subseteq\RR(I(w))\subseteq[e,w]_Bw^{-1}\subseteq\RR(C)$.
Thus all three sets are equal. For any $\gamma\in\Phi^+$,
\[
 \gamma\in\overline C
 \iff t_\gamma\in[e,w]_Bw^{-1}
 \iff t_\gamma w\le_Bw
 \iff\gamma\in I(w),
\]
where the last equivalence follows from the root-sign criterion. Hence $\overline C=I(w)=\ccl(C)$.
\end{proof}

\begin{proof}[Proof of Corollary~\ref{cor:joins}]
Set $C=\bigcup_{i\in J}E_i$.
Since $\pp\setminus C=\bigcap_{i\in J} (\pp\setminus E_i)$ is closed, $C$ is coclosed.
Since $C$ is contained in an inversion set, Theorem~\ref{thm:main} implies $\overline C=\ccl(C)=I(v)$ for some $v\in W$. Thus $\overline C$ is biclosed and is contained in every biclosed set containing $C$. Hence $\overline C=\bigvee_{i\in J} E_i$ in $\BC(\pp)$.
\end{proof}
\section*{Acknowledgment}
AI tools assisted in finding the proof and revising the manuscript; the authors developed the mathematical formulation and exposition. Y.G. is partially supported by NSFC Grant no. 12471309.
\bibliographystyle{plain}
\bibliography{ref}
\end{document}